\documentclass[12pt,twoside]{amsart}
\usepackage{amsmath}
\usepackage{amsfonts}
\usepackage{amssymb}
\usepackage{color}
\newcommand{\RR}{\mathbb R}

\newtheorem{theorem}{Theorem}[section]
\newtheorem{corollary}[theorem]{Corollary}
\newtheorem{lemma}[theorem]{Lemma}
\newtheorem{proposition}[theorem]{Proposition}
\newtheorem{definition}[theorem]{Definition}

\newcommand{\mysection}[1]{\section{#1} \setcounter{equation}{0}}
\newcommand{\e}{\varepsilon}
\begin{document}

\title[Convergence for a degenerate parabolic equation with gradient source]{Convergence to separate variables solutions for a degenerate parabolic equation with gradient source}

\author{Philippe Lauren\c cot}
\address{Institut de Math\'ematiques de Toulouse, CNRS UMR~5219, Universit\'e de Toulouse, F--31062 Toulouse Cedex 9, France}
\email{laurenco@math.univ-toulouse.fr}
\author{Christian Stinner}
\address{Fakult\"at f\"ur Mathematik, Universit\"at Duisburg-Essen, D--45117 Essen, Germany}
\email{christian.stinner@uni-due.de}

\date{\today}

\begin{abstract}
The large time behaviour of nonnegative solutions to a quasilinear degenerate diffusion equation with a source term depending solely on the gradient is investigated. After a suitable rescaling of time, convergence to a unique profile is shown for global solutions. The proof relies on the half-relaxed limits technique within the theory of viscosity solutions and on the construction of suitable supersolutions and barrier functions to obtain optimal temporal decay rates and boundary estimates. Blowup of weak solutions is also studied.
\end{abstract}

\maketitle

\mysection{Introduction}\label{sec:int}

Qualitative properties of nonnegative solutions to
\begin{eqnarray}
\partial_t u -\Delta_p u & = &  |\nabla u|^q\,, \qquad (t,x)\in Q:= (0,\infty)\times\Omega\,, \label{a0a}\\
u & = & 0\,, \qquad (t,x)\in (0,\infty)\times\partial\Omega\,, \label{a0b} \\
u(0) & = & u_0\,, \quad x\in\Omega\,, \label{a0c}
\end{eqnarray}
vary greatly according to the relative strength of the (possibly nonlinear and degenerate) diffusion $\Delta_p u := \mbox{ div }\left( |\nabla u|^{p-2}\ \nabla u \right)$ and the source term $|\nabla u|^q$ which is measured by the exponents $p\ge 2$ and $q>0$. More precisely, if $q\in (0,p-1)$, the comparison principle fails to be valid for the corresponding stationary equation \cite{BB01} and the existence of non-zero steady states is expected. The latter is known to be true for $p=2$ and $q\in (0,1)$ for a general bounded domain $\Omega$ \cite{BDD92,Lau07} and for $p>2$ and $q\in (0,p-1)$ if $\Omega=B(0,1)$ is the open unit ball of $\RR^N$ \cite{BLS10,Sti10}. A complete classification of nonnegative steady states seems nevertheless to be lacking in general, except in space dimension $N=1$ \cite{Lau07,Sti10} and when $\Omega=B(0,1)$ for radially symmetric solutions \cite{BLS10}. In these two particular cases, there is a one-parameter family $(w_\vartheta)_{\vartheta\in [0,1]}$ of stationary solutions to \eqref{a0a}-\eqref{a0b} with the properties $w_0=0$ and $w_\vartheta< w_{\vartheta'}$ in $\Omega$ if $\vartheta<\vartheta'$. In addition, each nonnegative solution to \eqref{a0a}-\eqref{a0c} converges as $t\to\infty$ to one of these steady states \cite{BLS10,Lau07,Sti10} and the available classification of the steady states plays an important role in the convergence proof. The classification of nonnegative steady states to \eqref{a0a}-\eqref{a0b} and the large time behaviour of nonnegative solutions to \eqref{a0a}-\eqref{a0c} thus remain unsolved problems when $q\in (0,p-1)$ and $\Omega$ is an arbitrary bounded domain of $\RR^N$, $N\ge 2$.

The situation is more clear for $q\ge p-1$ as the comparison principle \cite{BB01} guarantees that zero is the only stationary solution to \eqref{a0a}-\eqref{a0b}. Convergence to zero of nonnegative solutions to \eqref{a0a}-\eqref{a0c} is then expected in that case but the dynamics turn out to be more complicated as the gradient source term $|\nabla u|^q$ induces finite time blowup for some solutions. More precisely,  when $p=2$, global existence and convergence to zero for large times of solutions to \eqref{a0a}-\eqref{a0c} are shown in \cite{BDHL07,SZ06,Tch10} when either $q\in [1,2]$ or $q>2$ and $\|u_0\|_{C^1}$ is sufficiently small. The smallness condition on $u_0$ for $p=2$ and $q>2$ cannot be removed as finite time gradient blowup occurs for ``large'' initial data in that case \cite{So02}.  The blowup of the gradient then takes place on the boundary of $\Omega$ \cite{SZ06} and additional information on the blowup rate and location of the blowup points are provided in \cite{GH08,LS10}. In addition, the continuation of solutions after the blowup time is studied in \cite{BDL04} within the theory of viscosity solutions. Coming back to the convergence to zero of global solutions, still for $p=2$, the temporal decay rate and the limiting profile are identified in \cite{BDHL07} when $q\in (1,2]$ and shown to be that of the linear heat equation.

To our knowledge, the slow diffusion case $p>2$ has not been studied and the main purpose of this paper is to investigate what happens when $q\ge p-1$ and $p>2$. Our results may be summarized as follows: let $\Omega$ be a bounded domain of $\RR^N$ with smooth boundary $\partial\Omega$ (at least $C^2$) and consider an initial condition $u_0$ having the following properties:
\begin{equation}\label{a1}
u_0\in C_0(\bar{\Omega}) := \{ z \in C(\bar{\Omega})\ :\ z=0 \;\;\mbox{ on }\;\; \partial\Omega \}\,, \quad u_0\ge 0\,, \quad u_0\not\equiv 0\,.
\end{equation}
Then
\begin{description}
\item[(a)] if $q=p-1$, there is a unique global (viscosity) solution $u$ to \eqref{a0a}-\eqref{a0c} and $t^{1/(p-2)} u(t)$ converges as $t\to\infty$ to a unique profile $f$ which does not depend on $u_0$. In addition, $u_\infty: (t,x)\longmapsto t^{-1/(p-2)}\ f(x)$ is the unique solution to \eqref{a0a}-\eqref{a0b} with an initial condition identically infinite in $\Omega$, see Theorem~\ref{tha1} below. The availability of solutions having infinite initial value in $\Omega$ (also called \textit{friendly giants}) and their stability are well-known for the porous medium equation $\partial_t z - \Delta z^m=0$, $m>1$, the $p$-Laplacian equation $\partial_t z - \Delta_p z=0$, $p>2$, and some related equations sharing a similar variational structure, see \cite{MV94,RVPV09,Va07} for instance, but also for the semilinear diffusive Hamilton-Jacobi equation with gradient absorption $\partial_t z - \Delta z + |\nabla z|^q=0$, $q>1$ \cite{CLS89}.
\item[(b)] if $q\in (p-1,p]$, there is a unique global (viscosity) solution $u$ to \eqref{a0a}-\eqref{a0c} and $t^{1/(p-2)} u(t)$ converges as $t\to\infty$ to a unique profile $f_0$ which does not depend on $u_0$. In that case, $(t,x)\longmapsto t^{-1/(p-2)}\ f_0(x)$ is the unique solution to the $p$-Laplacian equation $\partial_t z - \Delta_p z=0$ with homogeneous Dirichlet boundary conditions and an initial condition identically infinite in $\Omega$, see Theorem~\ref{tha2} below. Therefore, the gradient source term $|\nabla u|^q$ does not show up in the large time dynamics.
\item[(c)] if $q>p$ and $u_0$ is sufficiently small, there is a unique global (viscosity) solution $u$ to \eqref{a0a}-\eqref{a0c} and $t^{1/(p-2)} u(t)$ converges as $t\to\infty$ to $f_0$ as in the previous case, see Theorem~\ref{tha2} below.
\item[(d)] if $q>p$ and $u_0$ is sufficiently large, then there is no global Lipschitz continuous weak solution to \eqref{a0a}-\eqref{a0c}, see Proposition~\ref{prz1} below. Let us point out that, since the notion of solution used for this result differs from that used for the previous cases, it only provides an indication that the smallness condition is needed in case~\textbf{(c)}.
\end{description}

Before stating precisely the main results, we point out that \eqref{a0a} is a quasilinear degenerate parabolic equation which is unlikely to have classical solutions. It turns out that a suitable framework for the well-posedness of the initial-boundary value problem \eqref{a0a}-\eqref{a0c} is the theory of viscosity solutions (see, e.g., \cite{BdCD97,Bl94,CIL92}) and we first define the notion of solutions to be used throughout this paper.

\begin{definition}\label{defa0}
Consider $u_0\in C_0(\bar{\Omega})$ satisfying \eqref{a1}. A function $u\in C([0,\infty)\times\bar{\Omega})$ is a solution to \eqref{a0a}-\eqref{a0c} if $u$ is a viscosity solution to \eqref{a0a} in $Q$ and satisfies
$$
u(t,x) = 0\,, \quad (t,x)\in (0,\infty)\times\partial\Omega \;\;\mbox{ and }\;\; u(0,x)=u_0(x)\,, \quad x\in\bar{\Omega}\,.
$$
\end{definition}

We begin with the case $p>2$ and $q=p-1$.

\begin{theorem}\label{tha1}
Assume that $p>2$, $q=p-1$, and consider $u_0\in C_0(\bar{\Omega})$ satisfying \eqref{a1}. Then, there is a unique solution $u$ to \eqref{a0a}-\eqref{a0c} in the sense of Definition~\ref{defa0} and
\begin{equation}\label{a5}
\lim_{t\to\infty} \left\| t^{1/(p-2)}\ u(t) - f \right\|_\infty = 0\,,
\end{equation}
where $f\in C_0(\bar{\Omega})$ is the unique positive solution to
\begin{equation}\label{a6}
- \Delta_p f - |\nabla f|^{p-1} - \frac{f}{p-2} =0 \;\;\mbox{ in }\;\; \Omega, \quad f>0 \;\;\mbox{ in }\;\;\Omega\,, \quad f=0 \;\;\mbox{ on }\;\; \partial\Omega\,.
\end{equation}
Moreover, $f\in W^{1,\infty}(\Omega)$.
\end{theorem}

Let us emphasize here that Theorem~\ref{tha1} not only gives a description of the large time behaviour of $u$, but also provides the existence and uniqueness of the positive solution $f$ to \eqref{a6}. To investigate the large time behaviour of $u$, no Liapunov functional seems to be available and we instead use the half-relaxed limits technique \cite{BP88,CIL92}. To this end, several steps are needed, including a comparison principle for \eqref{a6} which is established in Section~\ref{sec:acl} and upper bounds which guarantee on the one hand that the solutions to \eqref{a0a}-\eqref{a0c} decay at the expected temporal decay rate and on the other hand that there is no loss of boundary conditions as discussed for instance in \cite{BDL04}. The latter is achieved by the construction of suitable barrier functions. Also of importance is the positivity of the half-relaxed limits which allows us to apply the comparison lemma from Section~\ref{sec:acl}.

We next turn to the case $q>p-1$ and establish the following result.

\begin{theorem}\label{tha2}
Assume that $p>2$, $q>p-1$, and consider $u_0\in C_0(\bar{\Omega})$ satisfying \eqref{a1}. If $q>p$, assume further that
\begin{equation}\label{a1a}
u_0(x) \le \frac{f(x)}{\|\nabla f\|_\infty}\,, \qquad x\in\bar{\Omega}\,,
\end{equation}
where $f$ is the unique positive solution to \eqref{a6}. Then, there is a unique solution $u$ to \eqref{a0a}-\eqref{a0c} in the sense of Definition~\ref{defa0} and
\begin{equation}\label{a5b}
\lim_{t\to\infty} \left\| t^{1/(p-2)}\ u(t) - f_0 \right\|_\infty = 0\,,
\end{equation}
where $f_0\in C_0(\bar{\Omega})$ is the unique positive solution to
\begin{equation}\label{a7}
- \Delta_p f_0 - \frac{f_0}{p-2} = 0 \;\;\mbox{ in }\;\; \Omega, \quad f_0>0 \;\;\mbox{ in }\;\;\Omega\,, \quad f_0=0 \;\;\mbox{ on }\;\; \partial\Omega\,.
\end{equation}
\end{theorem}

For $q\in[p-1,p]$, the well-posedness of \eqref{a0a}-\eqref{a0c} easily follows from \cite{dL02} as already noticed in \cite{BDL04} for $p=2$. For $q>p$ and an initial condition $u_0$ satisfying \eqref{a1a}, it is a consequence of the Perron method and the comparison principle \cite{CIL92}. As for the large time behaviour, the existence and uniqueness of $f_0$ is shown in \cite{MV94} and the main contribution of Theorem~\ref{tha2} is the convergence \eqref{a5b}. The convergence proof follows the same lines as that of Theorem~\ref{tha1} but a new difficulty has to be overcome in the case $q=p$ for the boundary estimates. We also show that, when $q\in (p-1,p]$, powers of the positive solution $f$ to \eqref{a6} with an exponent in $(0,1]$ allow us to construct separate variables supersolutions to \eqref{a0a}-\eqref{a0b}.

Finally, the announced blowup result is proved in Section~\ref{sec:bu} by a classical argument \cite{HM04,QS07}.

\medskip

For further use, we introduce some notations: for $\xi\in\RR^N$ and $X\in\mathcal{S}(N)$, $\mathcal{S}(N)$ being the space of $N\times N$ real-valued symmetric matrices, we define the functions $F_0$ and $F$ by
\begin{eqnarray}
F_0(\xi,X) & := & - |\xi|^{p-2}\ tr(X) - (p-2)\ |\xi|^{p-4}\ \langle X \xi , \xi \rangle\,, \label{ham1} \\
F(\xi,X) & := & F_0(\xi,X) - |\xi|^q\,. \label{ham2}
\end{eqnarray}

\mysection{A comparison lemma}\label{sec:acl}

An important tool for the uniqueness of the positive solution to \eqref{a6} and the identification of the half-relaxed limits is the following comparison lemma between positive supersolutions and nonnegative subsolutions to the elliptic equation in \eqref{a6}.

\begin{lemma}\label{leb1}
Let $w\in USC(\bar{\Omega})$ and $W\in LSC(\bar{\Omega})$ be respectively a bounded upper semicontinuous (usc) viscosity subsolution and a bounded lower semicontinuous (lsc) viscosity supersolution to
\begin{equation}\label{b1}
- \Delta_p \zeta - |\nabla \zeta|^{p-1} - \frac{\zeta}{p-2} = 0 \;\;\mbox{ in }\;\; \Omega\,,
\end{equation}
such that
\begin{eqnarray}
w(x) = W(x) = 0 & \mbox{ for } & x\in\partial\Omega\,, \label{b2} \\
W(x) > 0 & \mbox{ for } & x\in\Omega\,. \label{b3}
\end{eqnarray}
Then
\begin{equation}\label{b4}
w \le W \;\;\mbox{ in }\;\; \bar{\Omega}\,.
\end{equation}
\end{lemma}

\begin{proof}
For $n\ge N_0$ large enough, $\Omega_n:= \{ x\in\Omega\ :\ d(x,\partial\Omega)>1/n\}$ is a non-empty open subset of $\Omega$. Since $\overline{\Omega_n}$ is compact and $W$ is lower semicontinuous, the function $W$ has a minimum in $\overline{\Omega_n}$ and the positivity \eqref{b3} of $W$ in $\overline{\Omega_n}$ implies that
\begin{equation}\label{b5}
\mu_n := \min_{\overline{\Omega_n}}{\{ W \}}>0\,.
\end{equation}
Similarly, the compactness of $\bar{\Omega}\setminus\Omega_n$ and the upper semicontinuity and boundedness of $w$ ensure that $w$ has at least one point of maximum $x_n$ in $\bar{\Omega}\setminus\Omega_n$ and we set
\begin{equation}\label{b6}
\eta_n := w(x_n) = \max_{\bar{\Omega}\setminus\Omega_n}{\{w\}} \ge 0\,,
\end{equation}
the maximum being nonnegative since $\partial\Omega\subset \bar{\Omega}\setminus\Omega_n$ and $w$ vanishes on $\partial\Omega$ by \eqref{b2}. We claim that
\begin{equation}\label{b7}
\lim_{n\to\infty} \eta_n=0\,.
\end{equation}
Indeed, owing to the compactness of $\bar{\Omega}$ and the definition of $\Omega_n$, there are $y\in\partial\Omega$ and a subsequence of $(x_n)_{n\ge N_0}$ (not relabeled) such that $x_n\to y$ as $n\to\infty$. Since $w(y)=0$, we deduce from the upper semicontinuity of $w$ that
$$
\lim_{\e\to 0}\ \sup{\{ w(x)\ :\ x\in B(y,\e)\cap\bar{\Omega} \}} \le 0\,.
$$
Given $\e>0$ small enough, there is $n_\e$ large enough such that $x_n\in B(y,\e)\cap\bar{\Omega}$ for $n\ge n_\e$ from which we deduce that
$$
0 \le \eta_n = w(x_n) \le \sup{\{ w(x)\ :\ x\in B(y,\e)\cap\bar{\Omega} \}}\,, \quad n\ge n_\e\,.
$$
Therefore,
$$
0 \le \limsup_{n\to\infty} \eta_n \le \sup{\{ w(x)\ :\ x\in B(y,\e)\cap\bar{\Omega} \}}\,,
$$
and letting $\e\to 0$ allows us to conclude that zero is a cluster point of $(\eta_n)_{n\ge N_0}$ as $n\to\infty$. The claim \eqref{b7} then follows from the monotonicity of $(\eta_n)_{n\ge N_0}$.

Now, fix $s\in (0,\infty)$. For $\delta>0$ and $n\ge N_0$, we define
\begin{eqnarray*}
z_n(t,x) & := & (t+s)^{-1/(p-2)}\ w(x) - s^{-1/(p-2)}\ \eta_n\,, \quad (t,x)\in [0,\infty)\times\bar{\Omega}\,, \\
Z_\delta(t,x) & := & (t+\delta)^{-1/(p-2)}\ W(x)\,, \quad (t,x)\in [0,\infty)\times\bar{\Omega}\,.
\end{eqnarray*}
It then follows from the assumptions on $w$ and $W$ that $z_n$ and $Z_\delta$ are respectively a bounded usc viscosity subsolution and a bounded lsc viscosity supersolution to
$$
\partial_t \zeta - \Delta_p \zeta - |\nabla \zeta|^{p-1} = 0 \;\;\mbox{ in }\;\; (0,\infty)\times\Omega\,,
$$
and satisfy
$$
Z_\delta(t,x) = 0 \ge - s^{-1/(p-2)}\ \eta_n = z_n(t,x)\,, \quad (t,x)\in (0,\infty)\times\partial\Omega\,.
$$
Moreover, if
\begin{equation}\label{b8}
0 < \delta < \left[ \frac{\mu_n}{1+\|w\|_\infty} \right]^{p-2}\ s\,,
\end{equation}
it follows from \eqref{b5} and \eqref{b8} that, for $x\in\Omega_n$,
$$
Z_\delta(0,x) = \delta^{-1/(p-2)}\ W(x) \ge \delta^{-1/(p-2)}\ \mu_n \ge s^{-1/(p-2)}\ \|w\|_\infty \ge z_n(0,x)\,,
$$
and from \eqref{b6} that, for $x\in\bar{\Omega}\setminus\Omega_n$,
$$
Z_\delta(0,x)\ge 0 \ge s^{-1/(p-2)}\ (w(x)-\eta_n) = z_n(0,x)\,.
$$
We are then in a position to apply the comparison principle \cite[Theorem~8.2]{CIL92} to conclude that
\begin{equation}\label{b9}
z_n(t,x)\le Z_\delta(t,x) \,, \qquad (t,x)\in [0,\infty)\times\bar{\Omega}\,,
\end{equation}
for any $\delta>0$ and $n\ge N_0$ satisfying \eqref{b8}. According to \eqref{b8}, the parameter $\delta$ can be taken to be arbitrarily small in \eqref{b9} from which we deduce that
$$
(t+s)^{-1/(p-2)}\ w(x) - s^{-1/(p-2)}\ \eta_n \le t^{-1/(p-2)}\ W(x)\,, \qquad (t,x)\in (0,\infty)\times\bar{\Omega}\,,
$$
for $n\ge N_0$. We next pass to the limit as $n\to\infty$ with the help of \eqref{b7} to conclude that
$$
(t+s)^{-1/(p-2)}\ w(x) \le t^{-1/(p-2)}\ W(x)\,, \qquad (t,x)\in (0,\infty)\times\bar{\Omega}\,.
$$
We finally let $s\to 0$ and take $t=1$ in the above inequality to obtain \eqref{b4}.
\end{proof}

A straightforward consequence of Lemma~\ref{leb1} is the uniqueness of the friendly giant.
\begin{corollary}\label{cob2}
There is at most one positive viscosity solution to \eqref{a6}.
\end{corollary}

Arguing in a similar way, we have a similar result for the $p$-Laplacian:
\begin{lemma}\label{leb3}
Let $w\in USC(\bar{\Omega})$ and $W\in LSC(\bar{\Omega})$ be respectively a bounded usc viscosity subsolution and a bounded lsc viscosity supersolution to
\begin{equation}\label{b10}
- \Delta_p \zeta - \frac{\zeta}{p-2} = 0 \;\;\mbox{ in }\;\; \Omega\,,
\end{equation}
satisfying \eqref{b2} and \eqref{b3}. Then $w \le W$ in $\bar{\Omega}$.
\end{lemma}

\mysection{Well-posedness: $q\in [p-1,p]$}\label{sec:wp1}

\begin{proposition}\label{pry1}
Assume that $q\in [p-1,p]$ and consider $u_0\in C_0(\bar{\Omega})$ satisfying \eqref{a1}. Then, there is a unique solution $u$ to \eqref{a0a}-\eqref{a0c} in the sense of Definition~\ref{defa0}.
\end{proposition}

\begin{proof}
Since the comparison principle holds true for \eqref{a0a}-\eqref{a0c} by \cite[Theorem~8.2]{CIL92}, Proposition~\ref{pry1} follows at once from \cite[Corollary~6.2]{dL02} provided that $\Sigma_-^p=\Sigma_+^p=(0,\infty)\times\partial\Omega$, where the sets $\Sigma_-^p$ and $\Sigma_+^p$ are defined as follows: denoting the distance $d(x,\partial\Omega)$ from $x\in\bar{\Omega}$ to $\partial\Omega$ by $d(x)$, $d$ is a smooth function in a neighbourhood of $\partial\Omega$ in $\bar{\Omega}$ and $(t,x)\in (0,\infty)\times\partial\Omega$ belongs to $\Sigma_-^p$ if
either
\begin{equation}\label{y1}
\liminf_{(y,\alpha)\to(x,0)} \left[ F\left( \frac{\nabla d(y) + o_\alpha(1)}{\alpha} , - \frac{\nabla d(y) \otimes \nabla d(y) + o_\alpha(1)}{\alpha^2} \right) + \frac{o_\alpha(1)}{\alpha} \right] > 0\,,
\end{equation}
or
\begin{equation}\label{y2}
\liminf_{(y,\alpha)\to(x,0)} \left[ F\left( \frac{\nabla d(y) + o_\alpha(1)}{\alpha} , \frac{D^2 d(y) + o_\alpha(1)}{\alpha} \right) + \frac{o_\alpha(1)}{\alpha} \right] > 0\,.
\end{equation}
Similarly,
$(t,x)\in (0,\infty)\times\partial\Omega$ belongs to $\Sigma_+^p$ if
either
\begin{equation}\label{y3}
\limsup_{(y,\alpha)\to(x,0)} \left[ F\left( \frac{-\nabla d(y) + o_\alpha(1)}{\alpha} , \frac{\nabla d(y) \otimes \nabla d(y) + o_\alpha(1)}{\alpha^2} \right) + \frac{o_\alpha(1)}{\alpha} \right] < 0\,,
\end{equation}
or
\begin{equation}\label{y4}
\limsup_{(y,\alpha)\to(x,0)} \left[ F\left( -\frac{\nabla d(y) + o_\alpha(1)}{\alpha} , -\frac{D^2 d(y) + o_\alpha(1)}{\alpha} \right) + \frac{o_\alpha(1)}{\alpha} \right] < 0\,.
\end{equation}

Now, consider $t>0$ and $x\in\partial\Omega$. For $y\in\bar{\Omega}$ sufficiently close to $x$ and $\alpha \in (0,1)$, we have
\begin{eqnarray*}
& & \alpha^p\ \left[ F\left( \frac{\nabla d(y) + o_\alpha(1)}{\alpha} , - \frac{\nabla d(y) \otimes \nabla d(y) + o_\alpha(1)}{\alpha^2} \right) + \frac{o_\alpha(1)}{\alpha} \right] \\
& = & |\nabla d(y) + o_\alpha(1)|^{p-2}\ (|\nabla d(y)|^2 + o_\alpha(1)) + (p-2)\ |\nabla d(y) + o_\alpha(1)|^{p-4}\ (|\nabla d(y)|^4 + o_\alpha(1)) \\
& & \ - \alpha^{p-q}\ |\nabla d(y) + o_\alpha(1)|^q + \alpha^{p-1}\ o_\alpha(1) \\
& = & (p-1)\ |\nabla d(y)|^p - \alpha^{p-q}\ |\nabla d(y)|^q + o_\alpha(1)\,.
\end{eqnarray*}
Since $|\nabla d(x)|=1$ and $p\ge q$, the right-hand side of the above inequality converges as $(y,\alpha)\to (x,0)$ either to $p-1$ if $q<p$ or to $p-2$ if $q=p$, both limits being positive since $p>2$. Therefore, the condition \eqref{y1} is satisfied so that $(t,x)$ belongs to $\Sigma_-^p$. Similarly,  for $y\in\bar{\Omega}$ sufficiently close to $x$ and $\alpha \in (0,1)$, we have
\begin{eqnarray*}
& & \alpha^p\ \left[ F\left( \frac{-\nabla d(y) + o_\alpha(1)}{\alpha} , \frac{\nabla d(y) \otimes \nabla d(y) + o_\alpha(1)}{\alpha^2} \right) + \frac{o_\alpha(1)}{\alpha} \right] \\
& = & - |\nabla d(y) + o_\alpha(1)|^{p-2}\ (|\nabla d(y)|^2 + o_\alpha(1)) - (p-2)\ |\nabla d(y) + o_\alpha(1)|^{p-4}\ (|\nabla d(y)|^4 + o_\alpha(1)) \\
& & \ - \alpha^{p-q}\ |\nabla d(y) + o_\alpha(1)|^q + \alpha^{p-1}\ o_\alpha(1) \\
& = & - (p-1)\ |\nabla d(y)|^p - \alpha^{p-q}\ |\nabla d(y)|^q + o_\alpha(1)\,,
\end{eqnarray*}
from which we readily infer that the condition \eqref{y3} is satisfied. Therefore, $(t,x)$ belongs to $\Sigma_+^p$ and we have thus shown that $\Sigma_-^P=\Sigma_+^p=(0,\infty)\times\partial\Omega$ as expected.
\end{proof}
\mysection{Large time behaviour: $q\in [p-1,p]$}\label{sec:ltb}

As already mentioned in the Introduction, the proofs of Theorems~\ref{tha1} and~\ref{tha2} involve several steps: we first show in the next section (Section~\ref{sec:ub}) that the temporal decay rate of $\|u(t)\|_\infty$ is indeed $t^{-1/(p-2)}$. To this end we construct suitable supersolutions which differ according to whether $q=p-1$ or $q>p-1$. In a second step (Section~\ref{sec:le}), we prove boundary estimates for large times which guarantee that no loss of boundary conditions occurs throughout the time evolution. Here again, we need to perform different proofs for $q\in [p-1,p)$ and $q=p$. The half-relaxed limits technique is then employed in Section~\ref{sec:cv} to show the expected convergence after introducing self-similar variables, and the existence of a positive solution $f$ to \eqref{a6} as well. The final result states that, if $u_0$ is bounded from above by $B\ f^\beta$ for some $B>0$ and $\beta\in (0,1]$, a similar bound holds true for $u(t)$ for positive times but with a possibly lower exponent $\beta$ (Section~\ref{sec:iub}).

\subsection{Upper bounds}\label{sec:ub}

\begin{lemma}\label{lec1} Assume that $q=p-1$. There is $C_1>0$ depending only on $p$, $q$, $\Omega$, and $\|u_0\|_\infty$ such that
\begin{equation}\label{c1}
u(t,x) \le C_1\ (1+t)^{-1/(p-2)}\,, \qquad (t,x)\in (0,\infty)\times\bar{\Omega}\,.
\end{equation}
\end{lemma}

\begin{proof}
Consider $x_0\not\in\bar{\Omega}$ and $R_0>0$ such that $\Omega\subset B(x_0,R_0)$. For $A>0$, $R>R_0$, $t\ge 0$, and $x\in\RR^N$, we put $r=|x-x_0|$,
$$
S(t,x) := A\ (1+t)^{-1/(p-2)}\ \sigma(r)\,, \qquad \sigma(r) := \frac{p-1}{p}\ \left( e^{pR/(p-1)} - e^{pr/(p-1)} \right)\,,
$$
and assume that
\begin{equation}\label{c2}
A \ge \max{\left\{ \left( \frac{e^{pR/(p-1)}}{(p-1)(p-2)} \right)^{1/(p-2)} , \frac{\|u_0\|_\infty}{\sigma(R_0)} \right\}}\,.
\end{equation}

Since $x_0$ does not belong to $\bar{\Omega}$, the function $S$ is $C^\infty$-smooth in $[0,\infty)\times\bar{\Omega}$ and, it follows from \eqref{c2} that, for $(t,x)\in Q$,
\begin{eqnarray*}
& & (1+t)^{(p-1)/(p-2)}\ \left\{ \partial_t S(t,x) + F(\nabla S(t,x), D^2 S(t,x)) \right\}\\
& = & -\frac{A}{p-2}\  \sigma(r) - A^{p-1}\ |\sigma'(r)|^{p-1} \\
& &\ - (p-1)\ A^{p-1}\ |\sigma'(r)|^{p-2}\ \sigma''(r) - (N-1)\ A^{p-1}\ \frac{|\sigma'(r)|^{p-2} \sigma'(r)}{r} \\
& = & A\ \left[  A^{p-2}\ \left( p-1+\frac{N-1}{r} \right)\ e^{pr} - \frac{\sigma(r)}{(p-2)} \right] \\
& \ge & A\ \left( (p-1)\ A^{p-2} - \frac{e^{pR/(p-1)}}{(p-2)} \right) \ge 0\,.
\end{eqnarray*}
Therefore, the condition \eqref{c2} guarantees that $S$ is a supersolution to \eqref{a0a} in $Q$. In addition, since $|x-x_0|<R_0<R$ for $x\in\Omega$, we have
$$
u(t,x) = 0 \le A\ (t+1)^{-1/(p-2)}\ \sigma(R_0) \le S(t,x)\,, \qquad  (t,x)\in (0,\infty)\times\partial\Omega\,,
$$
and
$$
u_0(x) \le \|u_0\|_\infty \le A\ \sigma(R_0) \le S(0,x)\,, \qquad x\in\bar{\Omega}\,,
$$
by \eqref{c2}. The comparison principle then implies that $u(t,x)\le S(t,x)$ for $(t,x)\in [0,\infty)\times \bar{\Omega}$, and Lemma~\ref{lec1} follows from this inequality.
\end{proof}

\begin{lemma}\label{lec2} Assume that $q>p-1$. There is $C_1>0$ depending only on $p$, $q$, $\Omega$, and $\|u_0\|_\infty$ such that
\begin{equation}\label{c3}
u(t,x) \le C_1\ (1+t)^{-1/(p-2)}\,, \qquad (t,x)\in (0,\infty)\times\bar{\Omega}\,.
\end{equation}
\end{lemma}

\begin{proof}
Consider $x_0\not\in\bar{\Omega}$ and $R_0>0$ such that $\Omega\subset B(x_0,R_0)$. For $A>0$, $\delta>0$, $R>R_0$, $t\ge 0$, and $x\in\RR^N$, we put $r=|x-x_0|$,
$$
S(t,x) := A\ (1+\delta t)^{-1/(p-2)}\ \varphi(r)\,, \qquad \varphi(r) := \frac{p-1}{p}\ \left( R^{p/(p-1)} - r^{p/(p-1)} \right)\,,
$$
and assume that
$$
A = \left( \frac{N}{2R_0^{q/(p-1)}} \right)^{1/(q-p+1)}\,, \quad R= \left( R_0^{p/(p-1)} + \frac{p \|u_0\|_\infty}{(p-1) A}\right)^{(p-1)/p}\,, \quad  \delta = \frac{N (p-2) A^{p-2}}{2 R^{p/(p-1)}} \,.
$$

Since $x_0$ does not belong to $\bar{\Omega}$, the function $S$ is $C^\infty$-smooth in $[0,\infty)\times\bar{\Omega}$ and, it follows from the properties $\Omega\subset B(x_0,R_0)$ and $q>p-1$ that, for $(t,x)\in Q$,
\begin{eqnarray*}
& & (1+\delta t)^{(p-1)/(p-2)}\ \left\{ \partial_t S(t,x) + F(\nabla S(t,x), D^2 S(t,x)) \right\}\\
& = & -\frac{A \delta}{p-2}\  \varphi(r) + N\ A^{p-1} - A^q\ (1+\delta t)^{-(q-p+1)/(p-2)}\ r^{q/(p-1)} \\
& \ge  & A^{p-1}\ \left[  N - A^{q-p+1}\ R_0^{q/(p-1)} - \frac{\delta R^{p/(p-1)}}{(p-2) A^{p-2}} \right] \\
& \ge & 0\,.
\end{eqnarray*}
Therefore, the function $S$ is a supersolution to \eqref{a0a} in $Q$ and the choice of $A$ and $R$ also guarantees that
$$
u_0(x) \le \|u_0\|_\infty \le A\ \varphi(R_0) \le S(0,x)\,, \qquad x\in\bar{\Omega}\,.
$$
Finally,
$$
u(t,x) = 0 \le A\ (1+\delta t)^{-1/(p-2)}\ \varphi(R_0) \le S(t,x)\,, \qquad  (t,x)\in (0,\infty)\times\partial\Omega\,,
$$
since $|x-x_0|<R_0<R$ for $x\in\Omega$ and we infer from the comparison principle that $u(t,x)\le S(t,x)$ for $(t,x)\in [0,\infty)\times \bar{\Omega}$. Lemma~\ref{lec2} then follows from this inequality.
\end{proof}

\subsection{Lipschitz estimates}\label{sec:le}

\begin{lemma}\label{lec3}
Assume that $q\in [p-1,p)$. Then there is $L_1>0$ such that
$$
|u(t,x)|=|u(t,x)-u(t,x_0)| \le \frac{L_1}{(1+t)^{1/(p-2)}}\ |x-x_0|\,, \quad (t,x,x_0)\in [1,\infty)\times\bar{\Omega}\times\partial\Omega\,.
$$
\end{lemma}

\begin{proof}
Since the boundary $\partial\Omega$ of $\Omega$ is smooth, it satisfies the uniform exterior sphere condition by \cite[Section~14.6]{GT01}, that is, there is $R_\Omega>0$ such that, for each $x_0\in\partial\Omega$, there is $y_0\in\RR^N$ satisfying $|x_0-y_0|=R_\Omega$ and $B(y_0,R_\Omega)\cap\Omega=\emptyset$.

We fix positive real numbers $A$, $M$, and $\delta$ such that
\begin{equation}\label{c4}
A := \max{\left\{ M , \frac{e C_1}{e-1} , \left( \frac{4 e^{p-1}}{p-2} \right)^{1/(p-2)} \right\}}\,, \quad M:= \frac{2^{1/(p-2)} \|u_0\|_\infty}{2^{1/(p-2)} - 1}\,,
\end{equation}
and
\begin{equation}\label{c5}
0< \delta < \min{\left\{ 1 , \frac{(p-2) R_\Omega}{N-1} , \left( \frac{1}{2 A^{q-p+1}} \right)^{1/(p-q)} \right\}}\,, \quad \Omega_\delta:= \{ x\in\Omega\ :\ d(x,\partial\Omega)>\delta \} \ne \emptyset\,,
\end{equation}
the constant $C_1$ being defined in Lemma~\ref{lec1} if $q=p-1$ and Lemma~\ref{lec2} if $q\in (p-1,p)$.

We next consider $t_0\ge 1$, $x_0\in\partial\Omega$, and let $y_0\in\RR^N$ be such that $|x_0-y_0|=R_\Omega$ and $B(y_0,R_\Omega)\cap\Omega=\emptyset$. We define the open subset $U_{\delta,x_0}$ of $\RR^N$ by
$$
U_{\delta,x_0} := \{ x\in \Omega\ : R_\Omega < |x-y_0| < R_\Omega+\delta \}\,,
$$
and the function
$$
S_{\delta,x_0}(t,x) := \frac{A}{(1+t)^{1/(p-2)}}\ \left( 1 - e^{-(|x-y_0|-R_\Omega)/\delta} \right) + \frac{M}{(1+t)^{1/(p-2)}} - \frac{M}{(1+t_0)^{1/(p-2)}}
$$
for $(t,x)\in [0,t_0]\times \overline{U_{\delta,x_0}}$. Since $y_0\not\in \overline{U_{\delta,x_0}}$, the function $S_{\delta,x_0}$ is $C^\infty$-smooth in $[0,t_0]\times \overline{U_{\delta,x_0}}$. For $(t,x)\in (0,t_0)\times U_{\delta,x_0}$, we set $r:=|x-y_0|-R_\Omega\in (0,\delta)$ and compute
\begin{eqnarray*}
& & \frac{(1+t)^{(p-1)/(p-2)}}{A^{p-1}}\ \left( \partial_t S_{\delta,x_0} - \Delta_p S_{\delta,x_0} - |\nabla S_{\delta,x_0}|^{q} \right)(t,x) \\
& = & - \frac{1- e^{-r/\delta}}{(p-2) A^{p-2}} - \frac{(N-1) e^{-(p-1)r/\delta}}{(r+R_\Omega) \delta^{p-1}} + \frac{(p-1) e^{-(p-1)r/\delta}}{\delta^p} \\
& &\ - \frac{e^{-qr/\delta}}{\delta^{q}}\ \frac{A^{q-p+1}}{(1+t)^{(q-p+1)/(p-2)}} - \frac{M}{(p-2) A^{p-1}}\\
& \ge &  \frac{e^{-(p-1)r/\delta}}{\delta^p}\ \left[ p-1- \frac{N-1}{r+R_\Omega}\ \delta - \frac{A^{q-p+1}\ \delta^{p-q}}{e^{(q-p+1)r/\delta}} - \frac{\delta^p e^{(p-1)r/\delta}}{(p-2) A^{p-2}} - \frac{M \delta^p e^{(p-1)r/\delta}}{(p-2) A^{p-1}} \right] \\
& \ge &  \frac{e^{-(p-1)r/\delta}}{\delta^p}\ \left[ p-1- \frac{N-1}{R_\Omega}\ \delta - A^{q-p+1}\ \delta^{p-q} - \frac{e^{p-1}}{(p-2) A^{p-2}} - \frac{M e^{p-1}}{(p-2) A^{p-1}} \right] \\
& \ge &  \frac{e^{-(p-1)r/\delta}}{\delta^p}\ \left[ 1 - A^{q-p+1}\ \delta^{p-q} - \frac{2e^{p-1}}{(p-2) A^{p-2}} \right] \ge 0\,,
\end{eqnarray*}
the last two inequalities being a consequence of the choice \eqref{c4} and \eqref{c5} of $\delta$, $A$, and $M$. Therefore, $S_{\delta,x_0}$ is a supersolution to \eqref{a0a} in $(0,\infty)\times U_{\delta,x_0}$. Moreover, since $t_0\ge 1$, we have
$$
S_{\delta,x_0}(0,x) \ge M -  \frac{M}{2^{1/(p-2)}} = \|u_0\|_\infty \ge u_0(x)\,, \quad x\in \overline{U_{\delta,x_0}}\,,
$$
by \eqref{c4}. It also follows from \eqref{c1} and \eqref{c3} that $u(t,x)\le C_1\ (1+t)^{-1/(p-2)}$ for $t\ge 0$ and $x\in\bar{\Omega}$. Then, if $(t,x)\in (0,t_0)\times \partial U_{\delta,x_0}$, either $x\in\partial\Omega$ and $u(t,x)=0\le S_{\delta,x_0}(t,x)$. Or $r=|x-y_0|-R_\Omega=\delta$ and it follows from \eqref{c4} that
$$
S_{\delta,x_0}(t,x) \ge \frac{A (1-e^{-1})}{(1+t)^{1/(p-2)}} \ge \frac{C_1}{(1+t)^{1/(p-2)}} \ge u(t,x)\,.
$$
We then deduce from the comparison principle \cite[Theorem~8.2]{CIL92} that $u(t,x)\le S_{\delta,x_0}(t,x)$ for $t\in [0,t_0]$ and $x\in\overline{U_{\delta,x_0}}$. In particular, for $t=t_0$,
$$
u(t_0,x) \le \frac{A}{(1+t_0)^{1/(p-2)}}\ \left( 1 - e^{-(|x-y_0|-R_\Omega)/\delta} \right)\,, \quad x\in \overline{U_{\delta,x_0}}\,.
$$
Consequently,
$$
0 \le u(t_0,x) - u(t_0,x_0) = u(t_0,x) \le \frac{A}{(1+t_0)^{1/(p-2)}}\ \left( 1 - e^{-(|x-y_0|-R_\Omega)/\delta} \right)\,, \quad x\in \overline{U_{\delta,x_0}}\,,
$$
whence, since $|x_0-y_0|-R_\Omega=0$,
\begin{equation}\label{c6}
0\le u(t_0,x) - u(t_0,x_0) \le \frac{A}{\delta (1+t_0)^{1/(p-2)}}\ |x-x_0|\,, \quad x\in\overline{U_{\delta,x_0}}\,.
\end{equation}

Consider finally $x\in\Omega$ and $x_0\in\partial\Omega$. If $|x-x_0| \ge\delta/2$, it follows from \eqref{c1} that
$$
|u(t_0,x)-u(t_0,x_0)| = u(t_0,x) \le \frac{2 C_1}{\delta (1+t_0)^{1/(p-2)}}\ |x-x_0|\,.
$$
If $|x-x_0|<\delta/2$, let $y_0\in\RR^N$ be such that $|x_0-y_0|=R_\Omega$ and $B(y_0,R_\Omega)\cap\Omega=\emptyset$. Since $x\in\Omega$, $|x-y_0|>R_\Omega$ and
$$
|x-y_0| \le |x-x_0|+|x_0-y_0| < R_\Omega+\delta\,.
$$
Consequently, $x\in U_{\delta,x_0}$ and we infer from \eqref{c6} that
$$
|u(t_0,x) - u(t_0,x_0)| \le \frac{A}{\delta (1+t_0)^{1/(p-2)}}\ |x-x_0|\,.
$$
 We have thus established Lemma~\ref{lec3} with $L_1:= \max{\{ 2 C_1 , A \}}/\delta$ for $(t,x,x_0)\in[1,\infty)\times \Omega \times \partial\Omega$. The extension to $[1,\infty)\times\bar{\Omega}\times\partial\Omega$ then readily follows thanks to the continuity of $u$ up to the boundary of $\Omega$.
\end{proof}

The previous proof does not apply to the case $q=p$ as the term $A^{q-p+1}\ \delta^{p-q}$ cannot be made arbitrarily small by a suitable choice of $\delta$.  Still, a similar result is valid for $q=p$ but first requires a change of variable.

\begin{lemma}\label{lec4}
Assume that $q=p$. Then there is $L_1>0$ such that
$$
|u(t,x)|=|u(t,x)-u(t,x_0)| \le \frac{L_1}{(1+t)^{1/(p-2)}}\ |x-x_0|\,, \quad (t,x,x_0)\in [1,\infty) \times \bar{\Omega} \times \partial\Omega\,.
$$
\end{lemma}

\begin{proof}
We define $h:=e^{u/(p-1)}-1$ and notice that
\begin{equation}\label{c7}
\frac{u}{p-1} \le h \le \frac{e^{u/(p-1)}}{p-1}\ u\,.
\end{equation}
By \eqref{a0a}-\eqref{a0c} and \cite[Corollaire~2.1]{Bl94} (or \cite[Proposition~2.5]{BdCD97}), $h$ is a viscosity solution to
\begin{eqnarray}
\partial_t \left[ \left( \frac{1+h}{p-1} \right)^{p-1} \right] - \Delta_p h & = & 0 \;\;\mbox{ in }\;\; Q\,, \label{c8a}\\
h & = & 0 \;\;\mbox{ on }\;\; (0,\infty)\times\partial\Omega\,, \label{c8b}\\
h(0) & = & e^{u_0/(p-1)}-1 \;\;\mbox{ in }\;\; \Omega\,. \label{c8c}
\end{eqnarray}

We fix positive real numbers $A$, $M$, and $\delta$ such that
\begin{equation}\label{c9}
A := \max{\left\{ 1 , M , \frac{e C_1}{(p-1) (e-1)}\ e^{C_1/(p-1)} \right\}}\,, \quad M:= \frac{2^{1/(p-2)} e^{\|u_0\|_\infty/(p-1)}}{2^{1/(p-2)} - 1}\,,
\end{equation}
and
\begin{equation}\label{c10}
0< \delta < \min{\left\{ 1 , \frac{(p-2) R_\Omega}{N-1} , \left( \frac{p-2}{2 e^{p-1}} \right)^{1/p}\ \left( \frac{3}{p-1} \right)^{-(p-2)/p} \right\}}\,, \quad \Omega_\delta \ne \emptyset\,,
\end{equation}
the constant $C_1$ and the set $\Omega_\delta$ being defined in Lemma~\ref{lec2} and \eqref{c5}, respectively.

We next consider $t_0\ge 1$, $x_0\in\partial\Omega$, and let $y_0\in\RR^N$ be such that $|x_0-y_0|=R_\Omega$ and $B(y_0,R_\Omega)\cap\Omega=\emptyset$, the definition of $R_\Omega$ and the existence of $y_0$ being stated at the beginning of the proof of Lemma~\ref{lec3}. We still define the open subset $U_{\delta,x_0}$ of $\RR^N$ by
$$
U_{\delta,x_0} := \{ x\in \Omega\ : R_\Omega < |x-y_0| < R_\Omega+\delta \}\,,
$$
and the function
$$
S_{\delta,x_0}(t,x) := \frac{A}{(1+t)^{1/(p-2)}}\ \left( 1 - e^{-(|x-y_0|-R_\Omega)/\delta} \right) + \frac{M}{(1+t)^{1/(p-2)}} - \frac{M}{(1+t_0)^{1/(p-2)}}
$$
for $(t,x)\in [0,t_0]\times \overline{U_{\delta,x_0}}$. Since $y_0\not\in \overline{U_{\delta,x_0}}$, the function $S_{\delta,x_0}$ is $C^\infty$-smooth in $[0,t_0]\times \overline{U_{\delta,x_0}}$. For $(t,x)\in (0,t_0)\times U_{\delta,x_0}$, we set $r:=|x-y_0|-R_\Omega\in (0,\delta)$ and compute
\begin{eqnarray*}
& & \frac{(1+t)^{(p-1)/(p-2)}}{A^{p-1}}\ \left( \partial_t \left[ \left( \frac{1+S_{\delta,x_0}}{p-1} \right)^{p-1} \right] - \Delta_p S_{\delta,x_0} \right)(t,x) \\
& = & - \frac{(1-e^{-r/\delta})}{(p-2) (p-1)^{p-2}}\ \frac{(1+S_{\delta,x_0})^{p-2}}{A^{p-2}} - \frac{M}{(p-2) (p-1)^{p-2}} \frac{(1+S_{\delta,x_0})^{p-2}}{A^{p-1}}\\
& & + \frac{(p-1) e^{-(p-1)r/\delta}}{\delta^p} - \frac{(N-1) e^{-(p-1)r/\delta}}{(r+R_\Omega) \delta^{p-1}} \\
& \ge &  \frac{e^{-(p-1)r/\delta}}{\delta^p}\ \left[ p-1- \frac{N-1}{R_\Omega}\ \delta - \frac{\delta^p\ e^{(p-1)r/\delta}}{(p-2) (p-1)^{p-2}}\ \left( \frac{1+2A}{A} \right)^{p-2} - \frac{M \delta^p e^{(p-1)r/\delta} (1+2A)^{p-2}}{(p-2) (p-1)^{p-2} A^{p-1}} \right] \\
& \ge &  \frac{e^{-(p-1)r/\delta}}{\delta^p}\ \left[ 1 - \frac{2 \delta^p\ e^{p-1}}{(p-2)} \ \left( \frac{3}{p-1} \right)^{p-2} \right] \ge 0\,,
\end{eqnarray*}
the last two inequalities being a consequence of the choice \eqref{c9} and \eqref{c10} of $\delta$, $A$, and $M$. Therefore, $S_{\delta,x_0}$ is a supersolution to \eqref{c8a} in $(0,\infty)\times U_{\delta,x_0}$. Moreover, since $t_0\ge 1$, we have
$$
S_{\delta,x_0}(0,x) \ge M -  \frac{M}{2^{1/(p-2)}} = e^{\|u_0\|_\infty/(p-1)} \ge h(0,x)\,, \quad x\in \overline{U_{\delta,x_0}}\,,
$$
by \eqref{c9}. It next follows from \eqref{c3} and \eqref{c7} that
\begin{equation}\label{c11}
h(t,x) \le \frac{e^{u(t,x)/(p-1)}}{p-1}\ u(t,x)\le \frac{C_1 e^{C_1/(p-1)}}{p-1}\ (1+t)^{-1/(p-2)}\,, \quad (t,x)\in [0,\infty)\times\bar{\Omega}\,.
\end{equation}
Then, if $(t,x)\in (0,t_0)\times \partial U_{\delta,x_0}$, either $x\in\partial\Omega$ and $h(t,x)=0\le S_{\delta,x_0}(t,x)$. Or $r=|x-y_0|-R_\Omega=\delta$ and it follows from \eqref{c9} and \eqref{c11} that
$$
S_{\delta,x_0}(t,x) \ge \frac{A (1-e^{-1})}{(1+t)^{1/(p-2)}} \ge \frac{C_1 e^{C_1/(p-1)}}{(p-1 )(1+t)^{1/(p-2)}} \ge h(t,x)\,.
$$
We then deduce from the comparison principle \cite[Theorem~8.2]{CIL92} that $h(t,x)\le S_{\delta,x_0}(t,x)$ for $t\in [0,t_0]$ and $x\in\overline{U_{\delta,x_0}}$. In particular, owing to \eqref{c7}, for $t=t_0$,
$$
\frac{u(t_0,x)}{p-1} \le h(t_0,x) \le \frac{A}{(1+t_0)^{1/(p-2)}}\ \left( 1 - e^{-(|x-y_0|-R_\Omega)/\delta} \right)\,, \quad x\in \overline{U_{\delta,x_0}}\,,
$$
and we argue as in the proof of Lemma~\ref{lec3} to complete the proof.
\end{proof}

We next proceed as in \cite[Theorem~5]{KK00} to deduce the Lipschitz continuity of $u(t)$ from Lemma~\ref{lec3} and Lemma~\ref{lec4}.

\begin{corollary}\label{coc5}
Assume that $q\in [p-1,p]$. Then there is $L_2>0$ such that
$$
|u(t,x)-u(t,y)| \le \frac{L_2}{(1+t)^{1/(p-2)}}\ |x-y|\,, \quad (t,x,y)\in [1,\infty) \times \bar{\Omega} \times \bar{\Omega}\,.
$$
\end{corollary}

\subsection{Convergence}\label{sec:cv}

Let $U$ be the solution to the $p$-Laplacian equation with homogeneous Dirichlet boundary conditions
\begin{eqnarray}
\partial_t U - \Delta_p U & = & 0\,, \qquad (t,x)\in Q\,, \label{x0a}\\
U & = & 0\,, \qquad (t,x)\in (0,\infty)\times\partial\Omega\,, \label{x0b} \\
U(0) &  = & u_0\,, \qquad x\in\Omega\,. \label{x0c}
\end{eqnarray}
Owing to the nonnegativity of $|\nabla u|^q$, the comparison principle \cite[Theorem~8.2]{CIL92} ensures that
\begin{equation}\label{a4a}
0 \le U(t,x) \le u(t,x)\,, \qquad (t,x)\in [0,\infty)\times\bar{\Omega}\,.
\end{equation}

We introduce the scaling variable $s=\ln{t}$ for $t>0$ and the new unknown functions $v$ and $V$ defined by
\begin{eqnarray}
u(t,x) & = & t^{-1/(p-2)}\ v(\ln{t},x)\,, \qquad (t,x)\in (0,\infty)\times\bar{\Omega}\,, \label{x1} \\
U(t,x) & = & t^{-1/(p-2)}\ V(\ln{t},x)\,, \qquad (t,x)\in (0,\infty)\times\bar{\Omega}\,, \label{x5}
\end{eqnarray}
Then $v$ is a viscosity solution to
\begin{eqnarray}
\partial_s v - \Delta_p v - e^{-(q-p+1)s/(p-2)}\ |\nabla v|^q - \frac{v}{p-2}& = & 0\,, \qquad (s,x)\in Q\,, \label{x2}\\
v & = & 0\,, \qquad (s,x)\in (0,\infty)\times\partial\Omega\,, \label{x3}\,, \\
v(0) &  = & u(1)\,, \qquad x\in\Omega\,. \label{x4}
\end{eqnarray}
In addition, owing to \eqref{c1} (if $q=p-1$), \eqref{c3} (if $q>p-1$), Corollary~\ref{coc5}, and \eqref{a4a}, we have
\begin{eqnarray}
V(s,x) \le v(s,x) & \le & C_1\,, \qquad (s,x)\in [0,\infty)\times\bar{\Omega}\,, \label{x6}\\
|v(s,x)-v(s,y)| & \le & L_2\ |x-y|\,, \qquad (s,x,y)\in [0,\infty)\times \bar{\Omega}\times \bar{\Omega}\,.\label{x7}
\end{eqnarray}

We next define for $\varepsilon\in (0,1)$
$$
w_\e(s,x) := v\left( \frac{s}{\e} , x \right) \,, \qquad (s,x)\in [0,\infty)\times\bar{\Omega}\,,
$$
and the half-relaxed limits
$$
w_*(x) := \liminf_{(\sigma,y,\e)\to (s,x,0)} w_\e(\sigma,y)\,, \qquad w^*(x) := \limsup_{(\sigma,y,\e)\to (s,x,0)} w_\e(\sigma,y)\,,
$$
for $(s,x)\in (0,\infty)\times\bar{\Omega}$. Observe that $w_*$ and $w^*$ are well-defined according to \eqref{x6} and indeed do not depend on $s>0$. In addition, it readily follows from \eqref{x3} and \eqref{x7} that
\begin{equation}\label{x8}
w_*(x)=w^*(x)=0\,, \quad x\in\partial\Omega\,.
\end{equation}
Also, $w_\e$ is a solution to
\begin{eqnarray}
& & \e\ \partial_s w_\e - \Delta_p w_\e - e^{-((q-p+1)s)/((p-2)\e)}\ |\nabla w_\e|^q - \frac{w_\e}{p-2} = 0 \;\;\mbox{ in }\;\; Q\,,\label{x9a} \\
& & w_\e = 0 \;\;\mbox{ on }\;\; (0,\infty)\times\partial\Omega\,, \label{x9b} \\
& & w_\e(0) = u(1) \;\;\mbox{ in }\;\; \Omega\,. \label{x9c}
\end{eqnarray}

\medskip

At this point, we distinguish the two cases $q=p-1$ and $q\in (p-1,p]$:

\medskip

\noindent\textbf{Case~1: $q=p-1$.} We use the stability of semicontinuous viscosity solutions \cite[Lemma~6.1]{CIL92} to deduce from \eqref{x9a} that
\begin{eqnarray}
& & w_* \;\mbox{ is a supersolution to \eqref{b1} in }\; \Omega\,, \label{x10} \\
& & w^* \;\mbox{ is a subsolution to \eqref{b1} in }\; \Omega\,. \label{x11}
\end{eqnarray}
In addition, as $V(s)\rightarrow f_0$ in $L^\infty(\Omega)$ as
$s\to\infty$ by \cite[Theorem~1.3]{MV94}, it also follows from
\eqref{x6} and the definition of $w_*$ and $w^*$ that
\begin{equation}\label{x12}
f_0(x) \le w_*(x) \le w^*(x) \le C_1\,, \qquad x\in\bar{\Omega}\,.
\end{equation}
Since $f_0>0$ in $\Omega$ by \cite[Theorem~1.1]{MV94}, we deduce from \eqref{x12} that $w_*$ and $w^*$ are positive and bounded in $\Omega$ and vanish on $\partial\Omega$ by \eqref{x8}. Owing to \eqref{x10} and \eqref{x11}, we are then in a position to apply Lemma~\ref{leb1} to conclude that $w^*\le w_*$ in $\bar{\Omega}$. Recalling \eqref{x12}, we have thus shown that $w_*=w^*$ in $\bar{\Omega}$. Setting $f:=w_*=w^*$, we infer from \eqref{x8}, \eqref{x10}, \eqref{x11}, and \eqref{x12} that $f\in C_0(\bar{\Omega})$ is a positive viscosity solution to \eqref{b1} so that it solves \eqref{a6}. We have thus proved the existence of a positive solution to \eqref{a6}, its uniqueness being granted by Corollary~\ref{cob2}. A further consequence of the equality $w_*=w^*$ is that $\|w_\e(1)-f\|_\infty\rightarrow 0$ as $\e\to 0$ (see, e.g., \cite[Lemme~4.1]{Bl94} or \cite[Lemma~5.1.9]{BdCD97}). In other words,
\begin{equation}\label{x13}
\lim_{s\to\infty} \|v(s)-f\|_\infty=0\,,
\end{equation}
which implies \eqref{a5} once written in terms of $u$.

Finally, a straightforward consequence of \eqref{x7} and \eqref{x13} is that $|f(x)-f(y)|\le L_2\ |x-y|$ for $(x,y)\in \bar{\Omega}\times\bar{\Omega}$. Consequently, $f$ is Lipschitz continuous in $\bar{\Omega}$ which completes the proof of Theorem~\ref{tha1}.

\medskip

\noindent\textbf{Case~2: $q\in (p-1,p]$.} We use once more the stability of semicontinuous viscosity solutions \cite[Lemma~6.1]{CIL92} to deduce from \eqref{x9a} that
\begin{eqnarray}
& & w_* \;\mbox{ is a supersolution to \eqref{b10} in }\; \Omega\,, \label{x14} \\
& & w^* \;\mbox{ is a subsolution to \eqref{b10} in }\; \Omega\,. \label{x15}
\end{eqnarray}
In addition, as $V(s)\rightarrow f_0$ in $L^\infty(\Omega)$ as $s\to\infty$ by \cite[Theorem~1.3]{MV94}, it also follows from \eqref{x6} and the definition of $w_*$ and $w^*$ that
\begin{equation}\label{x16}
f_0(x) \le w_*(x) \le w^*(x) \le C_1\,, \qquad x\in\bar{\Omega}\,.
\end{equation}
Since $f_0>0$ in $\Omega$ by \cite[Theorem~1.1]{MV94} and a solution to \eqref{b10}, we apply Lemma~\ref{leb3} to conclude that $w^*\le f_0$ in $\bar{\Omega}$. Recalling \eqref{x16}, we have proved that $w_*=w^*=f_0$ in $\bar{\Omega}$. We then complete the proof of Theorem~\ref{tha2} for $q\in (p-1,p]$ in the same way as that of Theorem~\ref{tha1}.

\subsection{Improved upper bounds}\label{sec:iub}

Interestingly, the positive solution $f$ to \eqref{a6} can be also used to construct supersolutions to \eqref{a0a}-\eqref{a0b} for $q>p-1$. We first consider the case $q\in (p-1,p]$ and postpone the case $q>p$ to Section~\ref{sec:wp2} where it is a crucial argument for the global existence of solutions.

\begin{proposition}\label{pre1}
Assume that $q\in (p-1,p]$ and there    are $\beta\in (0,1]$ and $B>0$ such that
\begin{equation}\label{e1}
u_0(x) \le B\ f(x)^\beta\,, \qquad x\in\bar{\Omega}\,.
\end{equation}
Then there is $\gamma\in (0,\beta]$ such that
\begin{equation}\label{e2}
u(t,x) \le \frac{\|f\|_\infty^{1-\gamma}}{\gamma \left( \|f\|_\infty^{p-2} + \gamma t \right)^{1/(p-2)}}\ f(x)^\gamma \le \frac{f(x)^\gamma}{\gamma \|f\|_\infty^\gamma}\,, \qquad (t,x)\in [0,\infty)\times \bar{\Omega}\,.
\end{equation}
\end{proposition}

\begin{proof}
We fix $\gamma\in (0,1)$ such
\begin{equation}\label{e3}
\gamma \le \min{\left\{ \frac{p-2}{p-1} , \beta , \frac{1}{B \|f\|_\infty^\beta} \right\}}\,,
\end{equation}
and, for $(t,x)\in [0,\infty)\times\bar{\Omega}$, we define
$$
\Sigma(t,x) = \frac{A f(x)^\gamma}{\gamma (1+\delta t)^{1/(p-2)}} \;\;\mbox{ with }\;\; A := \frac{1}{\|f\|_\infty^\gamma}\;\;\mbox{ and }\;\; \delta = \frac{\gamma}{\|f\|_\infty^{p-2}}\,.
$$
We claim that
\begin{equation}\label{e4}
\Sigma \;\mbox{ is a supersolution to \eqref{a0a} in }\; Q \;\mbox{ for }\; q\in [p-1,p]\,.
\end{equation}
Indeed, let $\phi\in C^2(Q)$ and consider $(t_0,x_0)\in Q$ where $\Sigma-\phi$ has a local minimum. Since $\Sigma$ is smooth with respect to the time variable, this property implies that
\begin{equation}\label{e5}
\partial_t \phi(t_0,x_0) = - \frac{\delta A}{\gamma (p-2)}\ \frac{f(x_0)^\gamma}{\left( 1+\delta t_0 \right)^{(p-1)/(p-2)}}\,,
\end{equation}
and that $x\mapsto \Sigma(t_0,x)-\phi(t_0,x)$ has a local minimum at $x_0$. In other words, the function
$x\mapsto f(x)^\gamma - \gamma\left( 1+\delta t_0 \right)^{1/(p-2)}\ \phi(t_0,x)/A$ has a local minimum at $x_0$ and we infer from \eqref{a6}, the positivity of $f$ in $\Omega$, and \cite[Corollaire~2.1]{Bl94} (or \cite[Proposition~2.5]{BdCD97}) that $g:=f^\gamma$ is a viscosity solution to
$$
-\Delta_p g - \frac{(1-\gamma)(p-1)}{\gamma}\ \frac{|\nabla g|^p}{g} - |\nabla g|^{p-1} - \frac{\gamma^{p-1}}{p-2}\ g^{(1-(1-\gamma)(p-1))/\gamma} = 0 \;\;\mbox{ in }\;\;\Omega\,.
$$
Consequently,
\begin{eqnarray*}
& & - \frac{\gamma^{p-1}}{A^{p-1}}\ \left( 1+\delta t_0 \right)^{(p-1)/(p-2)}\ \Delta_p\phi(t_0,x_0) - \frac{(1-\gamma)(p-1)\gamma^{p-1}}{A^p}\ (1+\delta t_0)^{p/(p-2)} \frac{|\nabla\phi(t_0,x_0)|^p}{f(x_0)^\gamma} \\
& & - \frac{\gamma^{p-1}}{A^{p-1}}\ \left( 1+\delta t_0 \right)^{(p-1)/(p-2)}\ |\nabla\phi(t_0,x_0)|^{p-1} - \frac{\gamma^{p-1}}{p-2}\ f(x_0)^{1-(1-\gamma)(p-1)}  \ge 0\,,
\end{eqnarray*}
from which we deduce, since $\gamma\in (0,1)$,
\begin{eqnarray}
- \Delta_p\phi(t_0,x_0) & \ge & \frac{(1-\gamma)(p-1)}{A}\ (1+\delta t_0)^{1/(p-2)} \frac{|\nabla\phi(t_0,x_0)|^p}{f(x_0)^\gamma} \label{e6} \\
& + & |\nabla\phi(t_0,x_0)|^{p-1} + \frac{A^{p-1}}{p-2}\ \frac{f(x_0)^{1-(1-\gamma)(p-1)}}{\left( 1+\delta t_0 \right)^{(p-1)/(p-2)}}  \,. \nonumber
\end{eqnarray}
By \eqref{e5} and \eqref{e6}, we have
\begin{equation}\label{e7}
\partial_t\phi(t_0,x_0) - \Delta_p\phi(t_0,x_0) - |\nabla\phi(t_0,x_0)|^q \ge \frac{|\nabla\phi(t_0,x_0)|^{p-1}}{f(x_0)^\gamma}\ R_1 + \frac{A^{p-1} f(x_0)^{1-(1-\gamma)(p-1)}}{\left( 1+\delta t_0 \right)^{(p-1)/(p-2)}}\ \frac{R_2}{p-2}\,,
\end{equation}
with
\begin{eqnarray*}
R_1 & := & \frac{(1-\gamma)(p-1)}{A}\ (1+\delta t_0)^{1/(p-2)} |\nabla\phi(t_0,x_0)| + f(x_0)^\gamma - f(x_0)^\gamma\  |\nabla\phi(t_0,x_0)|^{q-p+1}\,, \\
R_2 & := & 1 - \frac{\delta}{\gamma A^{p-2}}\ f(x_0)^{(1-\gamma)(p-2)}\,.
\end{eqnarray*}
On the one hand, \eqref{e3} guarantees that $(1-\gamma)(p-1)\ge 1$ which, together with Young's inequality and the assumption $q\in (p-1,p]$, leads us to
$$
R_1 \ge \|f\|_\infty^\gamma\ |\nabla\phi(t_0,x_0)| + f(x_0)^\gamma - (q-p+1)\ f(x_0)^\gamma\  |\nabla\phi(t_0,x_0)| - (p-q)\ f(x_0)^\gamma \ge 0\,.
$$
On the other hand, the choice of $A$ and $\delta$ gives
$$
R_2 = 1 - \left( \frac{f(x_0)}{\|f\|_\infty} \right)^{(1-\gamma)(p-2)} \ge 0\,.
$$
Combining the previous two inequalities with \eqref{e7} completes the proof of the claim \eqref{e4}.

Now, $u=\Sigma=0$ on $(0,\infty)\times\partial\Omega$ while, since $\beta\ge\gamma$, we infer from \eqref{e3} and the choice of $A$ that, for $x\in\bar{\Omega}$,
$$
u_0(x) \le B\ f(x)^\beta = \frac{A f(x)^\gamma}{\gamma}\ \frac{\gamma B f(x)^{\beta-\gamma}}{A} \le \Sigma(0,x)\ \frac{\gamma B \|f\|_\infty^{\beta-\gamma}}{A} \le \Sigma(0,x)\,.
$$
We then deduce from the comparison principle \cite[Theorem~8.2]{CIL92} that $u(t,x)\le\Sigma(t,x)$ for $(t,x)\in [0,\infty)\times\bar{\Omega}$ and the proof of Proposition~\ref{pre1} is complete.
\end{proof}

\mysection{Well-posedness and blowup: $q>p$}\label{sec:wpbu}

\subsection{Well-posedness}\label{sec:wp2}

We finally turn to the case $q>p$ and first show that a suitable multiple of the positive solution $f$ to \eqref{a6} allows us to construct a supersolution \eqref{a0a} when $q>p$ which vanishes identically on the boundary of $\Omega$.

\begin{lemma}\label{lef1}
Assume that $q>p-1$. Recalling that $f\in C_0(\bar{\Omega})$ is the unique positive solution to \eqref{a6}, the function
$$
\mathcal{F}(t,x) := \frac{f(x)}{\left( \|\nabla f\|_\infty^{p-2} + t \right)^{1/(p-2)}}\,, \qquad (t,x)\in [0,\infty)\times\bar{\Omega}\,,
$$
is a supersolution to \eqref{a0a} in $Q$.
\end{lemma}

\begin{proof}
Let $\phi\in C^2(Q)$ and consider $(t_0,x_0)\in Q$ where $\mathcal{F}-\phi$ has a local minimum. Since $\mathcal{F}$ is smooth with respect to the time variable and Lipschitz continuous with respect to the space variable, this property implies that
\begin{eqnarray}
\partial_t \phi(t_0,x_0) & = & - \frac{1}{p-2}\ \frac{f(x_0)}{\left( \|\nabla f\|_\infty^{p-2} + t_0 \right)^{(p-1)/(p-2)}}\,, \label{f1}\\
|\nabla\phi(t_0,x_0)| & \le & \frac{\|\nabla f\|_\infty}{\left( \|\nabla f\|_\infty^{p-2} + t \right)^{1/(p-2)}} \le 1\,, \label{f2}
\end{eqnarray}
and that $x\mapsto \mathcal{F}(t_0,x)-\phi(t_0,x)$ has a local minimum at $x_0$. In other words, the function
$x\mapsto f(x) - \left( \|\nabla f\|_\infty^{p-2} + t_0 \right)^{1/(p-2)}\ \phi(t_0,x)$ has a local minimum at $x_0$ and we infer from \eqref{a6} that
$$
- \left( \|\nabla f\|_\infty^{p-2} + t_0 \right)^{(p-1)/(p-2)}\ \Delta_p\phi(t_0,x_0) - \left( \|\nabla f\|_\infty^{p-2} + t_0 \right)^{(p-1)/(p-2)}\ |\nabla\phi(t_0,x_0)|^{p-1} - \frac{f(x_0)}{p-2} \ge 0\,,
$$
which, together with \eqref{f1}, gives
\begin{equation}\label{f3}
\partial_t \phi(t_0,x_0) - \Delta_p \phi(t_0,x_0) - |\nabla\phi (t_0,x_0)|^{p-1} \ge 0\,.
\end{equation}
We then infer from \eqref{f2}, \eqref{f3}, and the property $q>p-1$ that
$$
\partial_t \phi(t_0,x_0) - \Delta_p \phi(t_0,x_0) - |\nabla\phi (t_0,x_0)|^q \ge |\nabla\phi (t_0,x_0)|^{p-1} \left( 1 - |\nabla\phi (t_0,x_0)|^{q-p+1} \right) \ge 0\,,
$$
which completes the proof of Lemma~\ref{lef1}.
\end{proof}

\begin{proposition}\label{prf2}
Assume that $q>p$ and
\begin{equation}\label{f4}
u_0(x) \le \frac{f(x)}{\|\nabla f\|_\infty}\,, \qquad x\in\bar{\Omega}\,.
\end{equation}
Then there is a unique solution $u$ to \eqref{a0a}-\eqref{a0c} in the sense of Definition~\ref{defa0} and it satisfies
\begin{equation}\label{e2a}
u(t,x) \le \frac{f(x)}{\left( \|\nabla f\|_\infty^{p-2} + t \right)^{1/(p-2)}} \le \frac{f(x)}{\|\nabla f\|_\infty}\,, \qquad (t,x)\in [0,\infty)\times \bar{\Omega}\,.
\end{equation}
\end{proposition}

\begin{proof}
On the one hand, the solution $U$ to the $p$-Laplacian equation \eqref{x0a}-\eqref{x0c} is clearly a subsolution to \eqref{a0a} in $Q$. On the other hand, the function $\mathcal{F}$ defined in Lemma~\ref{lef1} is a supersolution to \eqref{a0a} in $Q$ by Lemma~\ref{lef1} and is thus also a supersolution to \eqref{x0a}. Since $U=\mathcal{F}=0$ on $(0,\infty)\times\partial\Omega$ and $U(0,x)=u_0(x)\le \mathcal{F}(0,x)$ for $x\in\bar{\Omega}$ by \eqref{f4}, the comparison principle \cite[Theorem~8.2]{CIL92} applied to the $p$-Laplacian equation \eqref{x0a} ensures that $U\le \mathcal{F}$ in $[0,\infty)\times\bar{\Omega}$. This property and the simultaneous vanishing of $U$ and $\mathcal{F}$ on $(0,\infty)\times\partial\Omega$ allow us to use the classical Perron method to establish the existence of a solution $u$ to \eqref{a0a}-\eqref{a0c} in the sense of Definition~\ref{defa0} which satisfies \eqref{e2a}. The uniqueness next follows from the comparison principle \cite[Theorem~8.2]{CIL92}.
\end{proof}

\subsection{Large time behaviour}\label{sec:ltb2}

We first recall that Lemma~\ref{lec2} is also valid in that case. It next readily follows from the Lipschitz continuity of $f$ and \eqref{e2a} that
$$
0 \le u(t,x) = u(t,x) - u(t,x_0) \le \frac{\|\nabla f\|_\infty}{(\|\nabla f\|_\infty^{p-2} + t)^{1/(p-2)}}\ |x-x_0|\,, \quad (t,x,x_0)\in [0,\infty)\times \bar{\Omega}\times \partial\Omega\,,
$$
and we proceed as in \cite[Theorem~5]{KK00} to show that Corollary~\ref{coc5} remains true (with a different constant $L_2$). The convergence proof is then the same as that performed in Section~\ref{sec:cv} for $q\in (p-1,p]$.

\subsection{Blowup}\label{sec:bu}

Let us first recall that, by a weak solution to \eqref{a0a}-\eqref{a0c}, we mean a nonnegative function $u\in C([0,\infty)\times\bar{\Omega})$ which belongs to $L^\infty(0,T; W^{1,\infty}(\Omega))$ and satisfies
\begin{equation}\label{z1}
\frac{d}{dt} \int_\Omega u(t,x)\ \psi(x)\ dx = \int_\Omega \left( - |\nabla u(t,x)|^{p-2}\ \nabla u(t,x) \cdot \nabla\psi(x) + |\nabla u(t,x)|^q\ \psi(x) \right)\ dx
\end{equation}
for any $\psi\in H^1_0(\Omega)$ and $T>0$. We now show that such a solution cannot exist for all times if $q>p$ and $u_0$ is sufficiently large.

\medskip

\begin{proposition}\label{prz1}
Assume that $q>p$ and define $r:=q/(q-p)$. There is a positive real number $\kappa$ depending on $\Omega$, $p$, and $q$ such that, if $\|u_0\|_{r+1}>\kappa$,  then \eqref{a0a}-\eqref{a0c} has no global weak solution.
\end{proposition}

\begin{proof}
We argue as in \cite[Theorem~2.4]{HM04} and use classical approximation arguments to deduce from \eqref{z1} and H\"older's and Young's inequalities that
\begin{eqnarray*}
\frac{1}{r+1}\ \frac{d}{dt}\|u\|_{r+1}^{r+1} & = & \int_\Omega u^r\ |\nabla u|^q\ dx - \frac{q}{q-p}\ \int_\Omega u^{r-1}\ |\nabla u|^p\ dx \\
& \ge & \int_\Omega u^r\ |\nabla u|^q\ dx - \frac{q}{q-p}\ |\Omega|^{(q-p)/q}\ \left( \int_\Omega u^r\ |\nabla u|^q\ dx \right)^{p/q} \\
& \ge & \int_\Omega u^r\ |\nabla u|^q\ dx - \frac{p}{q}\ \int_\Omega u^r\ |\nabla u|^q\ dx - \left( \frac{q}{q-p} \right)^{p/(q-p)}\ |\Omega| \\
& \ge & \frac{q-p}{q}\ \int_\Omega u^r\ |\nabla u|^q\ dx - \left( \frac{q}{q-p} \right)^{p/(q-p)}\ |\Omega|  \\
& \ge & \frac{q-p}{q}\ \left( \frac{q-p}{q-p+1} \right)^q\
\int_\Omega \left| \nabla\left( u^{(q-p+1)/(q-p)} \right)
\right|^q\ dx - \left( \frac{q}{q-p} \right)^{p/(q-p)}\ |\Omega|
\end{eqnarray*}
We now use the Poincar\'e inequality to obtain that
$$
\frac{1}{r+1}\ \frac{d}{dt}\|u\|_{r+1}^{r+1} \ge \kappa_1\ \int_\Omega u^{r+q}\ dx -\kappa_2
$$
for some constants $\kappa_1>0$ and $\kappa_2>0$ depending only on $\Omega$, $p$, and $q$. Since $q>1$, we use again H\"older's inequality to deduce
$$
\frac{1}{r+1}\ \frac{d}{dt}\|u\|_{r+1}^{r+1} \ge
\frac{\kappa_1}{|\Omega|^{(q-1)/(r+q)}}\ \|u\|_{r+1}^{r+q} -
\kappa_2\,.
$$
Since $q>1$, this clearly contradicts the global existence as soon as $\|u_0\|_{r+1}$ is sufficiently large.
\end{proof}

\subsection*{Acknowledgements}
The authors would like to thank Matteo Bonforte and Michael Winkler for helpful discussions and comments. This work was done during a visit of Ph.~Lauren\c{c}ot to the Fakult\"{a}t f\"{u}r Mathematik of the Universit\"{a}t Duisburg-Essen and while C.~Stinner held a one month invited position at the Institut de Math\'{e}matiques de Toulouse, Universit\'{e} Paul Sabatier - Toulouse III. We would like to express our gratitude for the invitation, support, and hospitality.



\end{document}